\documentclass{amsart}

\usepackage{amssymb}
\usepackage{graphicx}
\usepackage{MnSymbol}

\usepackage{enumerate}

\usepackage{amsmath,amscd}

\usepackage{amsthm}

\title{Bounded displacement permutations on tree-like spaces}

\author{Samuel M. Corson}
\address{E. T. S. I. I. Universidad Polit\'{e}cnica de Madrid, Jos\'{e} Guti\'{e}rrez Abascal 2, 28006 Madrid, Spain}
\email{sammyc973@gmail.com}

\theoremstyle{definition}\newtheorem{theorem}{Theorem}
\theoremstyle{definition}
\theoremstyle{definition}

\theoremstyle{definition}

\theoremstyle{definition}

\theoremstyle{definition}
\theoremstyle{definition}

\theoremstyle{definition}

\theoremstyle{definition}\newtheorem{corollary}[theorem]{Corollary}
\theoremstyle{definition}\newtheorem{proposition}[theorem]{Proposition}
\theoremstyle{definition}
\theoremstyle{definition}
\theoremstyle{definition}
\theoremstyle{definition}
\theoremstyle{definition}
\theoremstyle{definition}
\theoremstyle{definition}
\theoremstyle{definition}
\theoremstyle{definition}
\theoremstyle{definition}
\theoremstyle{definition}
\theoremstyle{definition}

\newcommand{\Sym}{\operatorname{Sym}}
\newcommand{\supp}{\operatorname{supp}}
\newcommand{\sep}{\operatorname{sep}}
\newcommand{\Bd}{\mathfrak{B}\mathfrak{d}}

\begin{document}



\thanks{The author is partially supported by RYC2023-045493-I.}

\begin{abstract}

It is shown that if a metric space exhibits certain finiteness and tree-like properties, then its group of bounded displacement does not include a subgroup isomorphic to the dyadic rational numbers.  This extends a result of N. M. Suchkov, A. A. Shlepkin, and D. A. Taysnyov.

\end{abstract}

\maketitle

\begin{section}{Introduction}

Let $\Sym(X)$ denote the symmetric group on the set $X$ (i.e. the group of all bijections on $X$).  We follow the convention that permutation groups act on the right.  If $(X, d)$ is a metric space we let $$\Bd(X, d) = \{g \in \Sym(X) \mid (\exists r_g \in \mathbb{Z}^+)(\forall x \in X) d(x, xg) \leq r_g\}$$ be the set of bounded displacement permutations (we'll write $\Bd(X)$ if the metric is clear).

Evidently $\Bd(X, d)$ is a subgroup of $\Sym(X)$ and is generally distinct from the group of isometries.  Multiplying by $-1$ gives an isometry on $\mathbb{Z}$ which is not in $\Bd(\mathbb{Z})$, and the transposition $(1 \hspace{.25cm} 2)$ is an element of $\Bd(\mathbb{Z})$ which is not an isometry.  If $(X_1, d_1)$ isometrically embeds into $(X_2, d_2)$ then the embedding gives a monomorphism of $\Bd(X_1, d_1)$ into $\Bd(X_2, d_2)$.  The group $\Bd(\mathbb{Z}^+)$ has been of interest (see e.g. \cite{SozSuchSuch}, \cite{SuchSuch}, \cite{SuchTaras}).  It has been shown \cite{SuchShlTaysn} that $\Bd(\mathbb{Z}^+)$ does not include a subgroup isomorphic to the group $\mathbb{Z}[1/2]$ of dyadic rational numbers, hence none isomorphic to $\mathbb{Q}$, answering \cite[Problem 20.98]{KhMaz}.  We extend this result to a more general setting.

We point out that the group $\Bd(\mathbb{Z}^+)$ is quite large.  Using transpositions of odd numbers with their successors, there is a subgroup of $\Bd(\mathbb{Z}^+)$ of exponent $2$ and cardinality $2^{\aleph_0}$.  The group $\Bd(\mathbb{Z}^+)$ also includes an isomorphic copy of a free (abelian) group of countably infinite rank and a copy of each countable locally finite group \cite{SozSuchSuch}.  From these one can easily construct a subgroup of $\Bd(\mathbb{Z}^+)$ isomorphic to $\mathbb{Z} \oplus \mathbb{Q}/\mathbb{Z}$, which feels quite close to $\mathbb{Q}$.  It is worth noting that if a metric space $(X, d)$ has a subset with finite diameter and infinitely many points, then $\Bd(X, d)$ includes a subgroup isomorphic to $\Sym(\mathbb{Z})$, and therefore includes an isomorphic copy of each countable group.

Let $B(x, r) = \{y \in X \mid d(y, x) \leq r\}$.  The main result of the paper is the following.

\begin{theorem}\label{maintheorem}  Suppose that for metric space $(X, d)$, simplicial tree $T$, and $f:X \rightarrow T$ there exist $\ell, m, q, \in \mathbb{Z}^+$ such that

\begin{enumerate}[(a)]

\item each vertex in $T$ has valence at most $\ell$;

\item $(\forall x \in X)\hspace{.25cm}|B(x, m^2 + m)| \leq q$; and

\item $(\forall x, y \in X) \hspace{.25cm} \frac{1}{m}d(x, y) - m \leq d_T(f(x), f(y)) \leq m \cdot d(x, y) + m$.

\end{enumerate}

\noindent Then the abelian group $\mathbb{Z}[1/2]$ of dyadic rationals is not isomorphic to a subgroup of $\Bd(X, d)$.

\end{theorem}

The hypotheses of Theorem \ref{maintheorem} essentially say that $X$ quasi-isometrically embeds into a simplicial tree of uniformly bounded valence, and balls in $X$ of specified radius (depending on the quasi-isometry constant) have uniformly bounded cardinality.  A finitely generated group which is virtually free satisfies these conditions.

\begin{corollary}\label{virtuallyfree}  If $H$ is finitely generated by set $S$ and virtually free then $\mathbb{Z}[1/2]$ is not isomorphic to a subgroup of $\Bd(H, d_{\Gamma})$.  (Here, $d_{\Gamma}$ is the metric on $H$ obtained from the Cayley graph $\Gamma = \Gamma(H, S)$.)
\end{corollary}

From this one can recover the earlier mentioned result of N. M. Suchkov, A. A. Shlepkin, and D. A. Taysnyov \cite{SuchShlTaysn} since the group $\Bd(\mathbb{Z})$ does not include a subgroup isomorphic to $\mathbb{Z}[1/2]$, hence $\Bd(\mathbb{Z}^+)$ does not either.

\end{section}

\begin{section}{Proof}

We give some definitions for a technical general statement.  For $g \in \Sym(X)$ we let $\supp(g) = \{x \in X \mid x \neq xg\}$ denote the \emph{support} of $g$.  Let $(X, d)$ be a metric space.  A sequence $(x_1, \ldots, x_k)$ in $X$ is an \textit{$r$-chain from $x_1$ to $x_k$} if for each $1 \leq i < k$ we have $d(x_i, x_{i + 1}) \leq r$.  Write $\sep(Y, r_1, r_2, x, y)$ if

\begin{itemize}

\item $Y \subseteq X$;

\item $x, y \in X$;

\item $r_1, r_2 \in \mathbb{Z}^+$;

\item if $d(x, z), d(y, w) \leq r_2$ then every $r_1$-chain from $z$ to $w$ contains an element of $Y$. 
\end{itemize}

\noindent We shall say that an element $g \in G$ \emph{is infinitely divisible by $2$ in $G$} if for each $n \in \mathbb{Z}^+$ there exists $h \in G$ such that $h^{2^n} = g$.  If $\sigma$ is a cycle in a permutation group let $|\sigma| = |\supp(\sigma)|$ in case $\sigma$ is finite, else write $|\sigma| = \infty$.

\begin{proposition}\label{moretechnical}  Suppose $(X, d)$ is a metric space such that

\begin{enumerate}
\item $(\forall r \in \mathbb{Z}^+)(\exists s \in \mathbb{Z}^+)(\forall x \in X) \hspace{.25cm}|B(x, r)| \leq s$; and

\item $(\forall r_1 \in \mathbb{Z}^+)(\exists t \in\mathbb{Z}^+)(\forall r_2 \in \mathbb{Z}^+)(\exists r_3 \in \mathbb{Z}^+)(\forall x, y \in X) \newline d(x, y) \geq r_3 \Rightarrow (\exists Y \subseteq X) \sep(Y, r_1, r_2, x, y) \wedge |Y| \leq t$.

\end{enumerate}

\noindent Then $\Bd(X, d)$ does not include a subgroup isomorphic to $\mathbb{Z}[1/2]$.
\end{proposition}

\begin{proof}
Assume the hypotheses and let $g_1$ be infinitely divisible by $2$ in $\Bd(X, d)$.  Let $d(x, xg_1) \leq r_1$ for all $x \in X$.  Select $t \in \mathbb{Z}^+$ as in assumption (2), without loss of generality $t \geq 2$.  As $g_1$ is infinitely divisible by $2$, select $n \in \mathbb{Z}^+$ and $g_2 \in \Bd(X, d)$ such that $n > 2t$ and $g_2^{2^n} = g_1$.  Let $d(x, xg_2) \leq r_2'$ for all $x \in X$ and let $r_2 = 2^n \cdot r_2'$.  Select $r_3 \in \mathbb{Z}^+$ as in assumption (2), without loss of generality $r_3 > r_2$.  By hypothesis (1) select $s \in \mathbb{Z}^+$ such that $|B(x, r_3)| \leq s$ for all $x \in X$.

Suppose for contradiction that cycles in the cyclic decomposition of $g_1$ have supports which are of arbitrarily large (possibly infinite) size.  Let $\sigma$ be a (possibly infinite) cycle in the decomposition of $g_1$ for which $|\sigma| \geq 2^n \cdot (s + 3)$.  Let $\tau$ be a cycle in the cyclic decomposition of $g_2$ such that $\supp(\tau) \supseteq \supp(\sigma)$.  Thus we have a cyclic decomposition $\tau^{2^n} = \sigma_1 \cdots \sigma_{2^{n'}}$, where $\sigma_1 = \sigma$ and in case $\sigma$ is infinite we have $2^{n'} = 2^n$ and in case $\sigma$ is finite we have $2^{n'} = \operatorname{gcd}(|\tau|, 2^n)$ and $|\sigma_i| = |\tau|/\operatorname{gcd}(|\tau|, 2^n)$ for each $1 \leq i \leq 2^{n'}$.  Selecting $x \in \supp(\sigma)$, we know since $|B(x, r_3)| \leq s$ that there is some $y \in \supp(\sigma) \setminus B(x, r_3)$.  Now $x, y \in \supp(\sigma) \subseteq \supp(\tau)$ and $d(x, y) \geq r_3$, so by assumption (2) there exists $Y \subseteq X$ such that $\sep(Y, r_1, r_2, x, y)$ and $|Y| \leq t$.  

Clearly for $1 \leq i \leq 2^{n'}$ and $u \in \supp(\tau)$ there exists $v \in \supp(\sigma_i)$ with $d(u, v) \leq n \cdot r_2'$.  Then for $1 \leq i \leq 2^{n'}$ there are $z_i, w_i \in \supp(\sigma_i)$ with $d(x, z_i), d(y, w_i) \leq n \cdot r_2' \leq 2^n \cdot r_2' = r_2$.  By applying $\sigma_i$ or $\sigma_i^{-1}$ we obtain an $r_1$-chain from $z_i$ to $w_i$, and so $Y \cap \supp(\sigma_i) \neq \emptyset$.  Now in case $2^{n'} \geq (2^n)^{\frac{1}{2}}$ we obtain a contradiction since then $2^{n'} \geq (2^n)^{\frac{1}{2}} > t \geq |Y|$ (that $2^n > t^2$ follows from the fact that $t \geq 2$ and $n > 2t$).  Thus we know that $\sigma$ is finite and that $2^{n'} = \operatorname{gcd}(|\tau|, 2^n) < (2^n)^{\frac{1}{2}}$, and in particular $|\sigma|$ is odd.  For brevity let $N = |\sigma|$.

In the cyclic group $\langle \sigma \rangle$ generated by the cycle $\sigma$, we know (since $N$ is odd) that for each $h_1 \in \langle \sigma \rangle$ there is a unique $h_2 \in \langle \sigma \rangle$ with $h_2^2 = h_1$.  In fact if $h_1 = \sigma^m$ with $0 \leq m < N$ we know $h_2 = \sigma^{m/2}$ if $m$ is even and $h_2 = \sigma^{N- \frac{N - m}{2}}$ in case $m$ is odd.  Define a sequence $m_1, m_2, m_3, \ldots$ of elements in $\{0, 1, \ldots, N - 1\}$ by letting $m_1 = N - \frac{N - 1}{2}$ and in case $m_i$ is even let $m_{i + 1} = \frac{m_i}{2}$ and in case $m_i$ is odd let $m_{i + 1} = N - \frac{N - m_i}{2}$.  Thus, $\sigma^{2m_1} = \sigma$ and $\sigma^{2m_{i + 1}} = \sigma^{m_i}$.  Since $2^{n'} =  \operatorname{gcd}(|\tau|, 2^n) < (2^n)^{\frac{1}{2}}$ we get $\tau^{2^{n - 1}}\upharpoonright \supp(\sigma) = \sigma^{m_1}$ and generally $\tau^{2^{n - i}}\upharpoonright \supp(\sigma) = \sigma^{m_i}$ for all $1 \leq i \leq t$.

Define a sequence $\lambda_1, \lambda_2, \ldots$ of real numbers by $\lambda_1 = \frac{1}{2}$ and $\lambda_{i + 1} = \frac{\lambda_i}{2}$ in case $m_i$ is even and $\lambda_{i + 1} = 1 - \frac{1 - \lambda_i}{2}$ in case $m_i$ is odd.  It is easy to see that $|m_i - N\lambda_i| \leq \frac{1}{2}$ for all $i \in \mathbb{Z}^+$ (by induction, the claim is true for $i = 1$ and in the successor case the check is easy).  Moreover we also have $0 < \lambda_i < 1$ and $\lambda_i$ is of form $\frac{w}{2^i}$ where $w$ is odd.  Then as $|m_i - N\lambda_i| \leq \frac{1}{2}$ and $s + 3 \leq \frac{N}{2^n}$ we see that the following sets are pairwise disjoint and all elements of the sets are less than $N$.

$$
\begin{array}{ll}
 \{0, 1, \ldots, s\} \\
\{m_1, m_1 + 1, \ldots, m_1 + s\} \\
\hspace{2cm}\vdots \\
\{m_t, m_t + 1, \ldots, m_t + s\}
\end{array}
$$

Returning to a point $x \in \supp(\sigma)$, note that the sequence $x, x\sigma, x\sigma^2, \ldots, x\sigma^s$ consists of pairwise distinct points and so there exists $0 < j \leq s$ such that $d(x, x\sigma^j) \geq r_3$.  No two of the following sequences of points has an element in common

$$
\begin{array}{ll}
 x, x\sigma, x\sigma^2,  \ldots, x\sigma^j \\
x\sigma^{m_1}, x\sigma^{m_1 + 1}, x\sigma^{m_1 + 2}, \ldots, x\sigma^{m_1 + j} \\
\hspace{2cm}\vdots \\
x\sigma^{m_t}, x\sigma^{m_t + 1}, x\sigma^{m_t + 2}, \ldots, x\sigma^{m_t + j}
\end{array}
$$

\noindent Also, we know by the triangle inequality that 

$$
\begin{array}{ll}
 d(x, x\sigma^{m_1}) & \leq d(x, x\tau) + d(x\tau, x\tau^2) + \cdots + d(x\tau^{2^{n-1} - 1}, x\tau^{2^{n-1}}) \\
& \leq 2^{n-1} \cdot r_2'  \\
&   \leq 2^n \cdot r_2' = r_2
\end{array}
$$

\noindent and similarly $d(x, x\sigma^{m_i}) \leq 2^{n - i}\cdot r_2' \leq r_2$ for each $1 \leq i \leq t$.  By the same token, $d(x\sigma^j, x\sigma^{m_i + j}) \leq r_2$ for each $1 \leq i \leq t$.  So selecting a set $Y'$ with $\sep(Y', r_1, r_2, x, x\sigma^j)$ and $|Y'| \leq t$ we once again obtain a contradiction.  So, the cycles in a decomposition of $g_1$ have a uniform bound on their size, so $g_1$ is torsion.
\end{proof}

\begin{proof}[Proof of Theorem \ref{maintheorem}]
Assume the hypotheses.  Without loss of generality $\ell \geq 2$.  Let $T^0$ denote the set of vertices of the simplicial tree $T$.  For a vertex $v \in T^0$ and $R \in \mathbb{Z}^+$ it is clear that $$|B(v, R) \cap T^0| \leq \ell^0 + \ell^1 + \cdots + \ell^R \leq \ell^{R + 1}.$$  If $p \in T$ is an arbitrary point and $R \in \mathbb{Z}^+$ we take $v \in T^0$ with $d_T(v, p) \leq \frac{1}{2}$ and let $V = \{v' \in B(v, R + 1) \cap T^0 \mid f(X) \cap B(v', \frac{1}{2}) \neq \emptyset\}$.   For each $v' \in V$ select $y_{v'} \in X$ such that $d_T(f(y_{v'}), v') \leq \frac{1}{2}$.  It is clear that if $f(x) \in B(p, R) \subseteq  B(v, R+1)$ there is some $v' \in V$ with $d_T(f(x), v') \leq \frac{1}{2}$.  Then $$f^{-1}(B(p, R)) \subseteq \bigcup_{v' \in V}f^{-1}(B(f(y_{v'}), 1)) \subseteq \bigcup_{v' \in V} B(y_{v'}, m^2 + m)$$ and so $|f^{-1}(B(p, R))| \leq \ell^{R + 2} \cdot q$.

We verify the conditions of Proposition \ref{moretechnical}.  Towards condition (1) let $x \in X$ and $r \in \mathbb{Z}^+$.  Certainly $f(B(x, r)) \subseteq B(f(x), m \cdot r + m)$, and as $$B(x, r) \subseteq f^{-1}(f(B(x, r))) \subseteq f^{-1}(B(f(x), m \cdot r + m))$$ we get $|B(x, r)| \leq \ell^{m \cdot r + m + 2} \cdot q$.

To verify condition (2) let $r_1 \in \mathbb{Z}^+$ be given.  Let $t = \ell^{m \cdot r_1 + m + 2} \cdot q$.  Given $r_2 \in \mathbb{Z}^+$ we let $r_3 = m^2(2r_1 + 2r_2 + 3) $.  Suppose $d(x, y) \geq r_3$.  Let $\gamma$ be the geodesic path in $T$ between $f(x)$ and $f(y)$.  As $d_T(f(x), f(y)) \geq \frac{1}{m}d(x, y) - m \geq 2m(r_1 + r_2 + 1)$ we can select a point $p$ in the image of $\gamma$ such that both $d_T(p , f(x))$ and $d_T(p, f(y))$ are at least $m( r_1 +  r_2 + 1)$.  Let $Y = f^{-1}(B(p, m\cdot r_1 + m))$ and note that $|Y| \leq \ell^{m\cdot r_1 + m + 2} \cdot q$.

Let $z \in B(x, r_2)$, $w \in B(y, r_2)$, and $(z = x_1, x_2, \ldots, x_k = w)$ be an $r_1$-chain from $z$ to $w$.  Note that both $d_T(f(z), f(x))$ and $d_T(f(w), f(y))$ are at most $m \cdot r_2 + m$ and for each $1 \leq i < k$ we have $d_T(f(x_i), f(x_{i + 1})) \leq m \cdot r_1 + m$.  Let $\gamma_i$ be the geodesic path from $f(x_i)$ to $f(x_{i + 1})$, $\gamma_{xz}$ be that from $f(x)$ to $f(z)$, and $\gamma_{w, y}$ be that from $f(w)$ to $f(y)$.  The concatenation $\gamma' = \gamma_{x, z}\gamma_1\gamma_2 \cdots\gamma_{k-1}\gamma_{w, y}$ gives a path in $T$ from $f(x)$ to $f(y)$, so the image of $\gamma$ is a subset of the image of $\gamma'$.  Thus $p$ is in the image of $\gamma'$.  If $p$ is in the image of $\gamma_{x, z}$ then $d_T(p, f(x)) \leq m \cdot r_2 + m$, which is a contradiction, and similarly $p$ is not in the image of $\gamma_{w, y}$.  Thus $p \in \gamma_i$ for some $i$, and so $f(x_i) \in B(p, m \cdot r_1 + m)$, and $x_i \in Y$.  We conclude that $\sep(Y, r_1, r_2, x, y)$ and condition (2) holds.
\end{proof}

\end{section}

\section*{Acknowledgement}

The author thanks Anton A. Klyachko for pointing out \cite{Klya} that \cite[Problem 20.98]{KhMaz} was already answered in \cite{SuchShlTaysn}.  The author also thanks the reviewers for carefully reading the paper.  Heartfelt thanks are due especially to reviewer B for pointing out a sizable gap in the original argument of Proposition \ref{moretechnical}, which lead to a revision of the contents of this article and more modest claims in the main theorem.

\end{document}